\documentclass[11pt,reqno]{amsart}
\usepackage{fullpage}
\usepackage{amsmath}
\usepackage{amsfonts}
\usepackage{amssymb}
\usepackage{graphicx}
\usepackage{mathrsfs}
\usepackage{epsfig}
\usepackage{amsthm}
\usepackage{verbatim}
\usepackage{graphicx}
\usepackage{color}
\usepackage{url}

\usepackage[colorlinks,citecolor=black,linkcolor=black,
            bookmarksopen,
            bookmarksnumbered
           ]{hyperref}

\newcommand{\black}{\color{black}}

\newcommand{\beq}{\begin{equation}}
\newcommand{\eeq}{\end{equation}}
\newcommand{\e}{\mathrm{e}}
\newcommand{\dd}{\mathrm{d}}

\newcommand{\taul}{\tau}
\newcommand{\vta}{u_{\tau}}
\newcommand{\uta}{u_{\tau}}

\newcommand{\unum}[1]{u^{#1}}
\newcommand{\un}{\unum{n}}

\newtheorem{theorem}{Theorem}[section]

\newtheorem{lemma}[theorem]{Lemma}
\newtheorem{cor}[theorem]{Corollary}
\newtheorem{proposition}[theorem]{Proposition}

\allowdisplaybreaks[3]

\author{Alexander Ostermann}
\address{Department of Mathematics, University of Innsbruck, Technikerstr.~13, 6020 Innsbruck, Austria (A. Ostermann)}
\email{alexander.ostermann@uibk.ac.at}

\author{Fr\'ed\'eric Rousset}
\address{Universit\'e Paris-Saclay, CNRS,   Laboratoire de Math\'ematiques d'Orsay (UMR 8628),  91405 Orsay Cedex, France (F. Rousset)}
\email{frederic.rousset@universite-paris-saclay.fr}

\author{Katharina Schratz}
\address{LJLL (UMR 7598), Sorbonne Universit\'e, UPMC, 4 place Jussieu, 75005, Paris, France (K. Schratz)}
\email{katharina.schratz@ljll.math.upmc.fr}

\begin{document}
\begin{abstract}
We study  a filtered Lie splitting scheme  for the cubic nonlinear Schr\"{o}dinger equation. 
  We establish error estimates at low regularity by using   discrete Bourgain spaces. This allows us to handle data in $H^s$ with $0<s<1$ overcoming the standard stability restriction to smooth Sobolev spaces with index $s>1/2$ . More precisely, we prove convergence rates of order $\tau^{s/2}$ in $L^2$ at this level of regularity.
\end{abstract}

\title[]{Error estimates at low regularity \\ of splitting schemes for NLS}

\maketitle

\section{Introduction}

We consider the cubic nonlinear Schr\"{o}dinger equation (NLS)
\begin{equation}\label{nlsO}
\begin{aligned}
i \partial_t u &= - \partial_{x}^2 u -\mu  \vert u \vert^2 u, \qquad (t,x) \in \mathbb{R} \times {\Omega}
\end{aligned}
\end{equation}
equipped with periodic boundary conditions  $\Omega = \mathbb{T}$ or  set on the full space $\Omega = \mathbb{R}$. We allow  both the defocusing and focusing case $\mu = \pm 1$. In the last decades, a large variety of numerical schemes has been proposed to approximate the time dynamics of NLS. A particular attractive class of schemes are splitting methods, where the right-hand side  is split into the linear $T(u) = i \partial_{x}^2 u$  and nonlinear part $V(u) = i\mu\vert u \vert^2 u$, respectively. This splitting approach is  commonly used in practical computations as it  is easy to implement and preserves the mass in the system, i.e., the $L^2$ norm of the solution.   For smooth solutions the error behaviour of splitting methods is nowadays well  understood. If the initial value satisfies $u(0) \in H^2$ then the Lie splitting scheme
\begin{equation}\label{LieCl}
\begin{aligned}
u_{\text{Lie}}^{n+1}  &  = \e^{i \tau \partial_x^2} \left( \e^{i \mu \tau \vert  u_{\text{Lie}}^{n} \vert^2} u_{\text{Lie}}^{n} \right), \qquad u_{\text{Lie}}^0  = u(0)
\end{aligned}
\end{equation}
allows for global error estimates at order $\tau$ in $L^2$; see \cite{ESS16,Faou12,Ignat11,Lubich08}.   This  smoothness requirement, i.e., that  $u(0) \in H^2$, stems from the local approximation error. More precisely, one can show  that the  Lie splitting method  \eqref{LieCl} applied to  cubic NLS \eqref{nlsO} introduces a local error  of type $\tau^2 [T,V](u)$  with the commutator given by  (see \cite[Section 4.2]{Lubich08})
\begin{equation}\label{clELie}
 \frac{1}{2\mu} [T,V](u) =  \overline u \left(\partial_x u\right)^2 + 2u \left(\partial_x u\right)\left(\partial_x \overline u\right) + u^2 \partial_{xx} \overline u.
\end{equation}
From the above relation we readily see that  the last term $u^2 \partial_{xx} \overline u$ requires the boundedness of two additional derivatives of the solution such that global first-order convergence in $L^2$ will require (at least)  $H^2$ solutions. Indeed,  a classical Lady Windermere  argument (see \cite{HNW}) combined with the standard bilinear estimate
\begin{equation*}\label{bi}
\Vert v w \Vert_{H^\sigma} \leq C_\sigma \Vert v  \Vert_{H^\sigma}\Vert w \Vert_{H^\sigma}, \qquad \sigma > 1/2
\end{equation*}
 allows us to  deduce from the local error structure \eqref{clELie} global first-order convergence in $H^\sigma$ for solutions $u(t) \in H^{\sigma + 2}$ for any $\sigma > 1/2$. Using a refined global error analysis, by first proving fractional convergence of the scheme in a suitable higher order Sobolev space (which implies a priori the boundedness of the numerical solution in this space \cite{Lubich08}), one can push down the error analysis   to $L^2$ with a global error of order $\tau$ for $H^{2}$ solutions.

One may now wonder what happens for less smooth solutions  $u(t) \in H^s$ with $s < 2$. Clearly, we can no longer expect full first-order convergence due to the local error structure \eqref{clELie} which involves the term  $u^2 \partial_{xx} \overline u$. However, one may expect to achieve fractional error estimates with the rate of convergence depending  on the regularity of the solution. Indeed, with a refined global error analysis, fractional convergence of order $\tau^{s/2}$  can be established in $L^2$  for solutions in $H^s$ as long as   $s>1/2$. The latter restriction   $s>1/2$ is thereby crucial when estimating the nonlinear terms in the global error  by a Sobolev embedding theorem
\begin{equation}\label{sob}
\Vert  v w \Vert_{L^2 } \leq C \Vert v  \Vert_{L^2}\Vert w \Vert_{H^{r}},\qquad r>\tfrac12.
\end{equation}
The  nonlinear $L^2$ estimate \eqref{sob}  restricts  the class of possible solutions   to smooth Sobolev spaces $u(t) \in H^{s}$ with  $s>1/2$. Classical arguments based on  the nonlinear $L^2$ estimate \eqref{sob}   break down  for rough solutions.  For a long time it was therefore an open question whether  one can  actually establish global convergence (with arbitraryly small order of convergence $\tau^{\varepsilon(s)}$, $\varepsilon>0$) for splitting methods for NLS under rough data
\begin{equation}\label{rough}
u(0) \in H^{s}, \qquad 0< s \le 1/2.
\end{equation}
The aim of this paper lies in closing this gap.

In this paper we establish low regularity estimates for the filtered Lie splitting method
\begin{equation}\label{scheme}
\begin{aligned}
u^{n+1}  &  = \e^{i \tau \partial_x^2}  \Pi_{\taul}\left( \e^{i \mu\tau \vert \Pi_{\taul} u^n \vert^2} \Pi_{\taul} u^n \right),
\qquad u^0 & =\Pi_{\taul}  u(0)
\end{aligned}
\end{equation}
with the aid of discrete Bourgain spaces; cf. \cite{ORSBourg}.  This will allow us for the first time to deal with rough data \eqref{rough}.
Here, the projection operator $ \Pi_{\taul}$ for $\tau>0$ is defined by the Fourier multiplier
\beq
\label{PiKdef}
 \Pi_{\taul} =  \chi\left({ -i \partial_x \over\taul^{-\frac12} }\right) = \overline\Pi_{\taul},
\eeq
where $\chi$ is the characteristic function $\chi= \mathrm{1}_{[-1, 1]}.$ Note that the projection  $\Pi_{\tau}$
is then continuous on $L^p$ for $p \in (1, + \infty)$ with operator norm uniformly bounded for $\tau \in (0, 1]$.
Such a filtered splitting was 
\black
originally proposed for nonlinear Schr\"{o}dinger equations on the full space by Ignat \cite{Ignat11}; see also Ignat and Zuazua \cite{IZ09}.
 
For the filtered Lie splitting~\eqref{scheme} first order error estimates for $H^2(\mathbb{R}^d)$ solutions were  established in  \cite{Ignat11}  with the aid of discrete Strichartz type estimates  $L^2(\mathbb{R}^d)$  in  dimensions $d \leq 3$.   Splitting methods with  numerical filters have been also successfully introduced  in  \cite{Bao} for nonlinear Schr\"odinger equations  in the semiclassical regime  with attractive interaction to suppress numerically the modulation instability. 

The main result of the paper is to prove that the filtered Lie splitting \eqref{scheme} indeed allows approximations to the cubic  NLS  \eqref{nlsO} with rough initial data \eqref{rough}. More precisely, we show that

\begin{theorem}\label{maintheo}
For every $T >0$ and $u_{0} \in H^{s_{0}}$,  $s_{0}>0$, let $u \in \mathcal{C}([0, T], H^{s_{0}})$ be the exact solution of \eqref{nlsO} with initial datum $u_{0}$ and denote by $\un$ the sequence defined by the scheme~\eqref{scheme}. Then, for every $0<s_{1}<s_{0}$, $s_{1} \in [0, 2]$, we have  the following error estimate:
there exists $\tau_{0}>0$ and $C_{T}>0$ such that for every step size $\tau \in (0, \tau_{0}]$
\beq
\label{final1} \|\un- u(t_{n})\|_{L^2} \leq C_{T} \tau^{s_{1}\over 2}, \quad 0 \leq n\tau \leq T.
\eeq
\end{theorem}

With the aid of  discrete Strichartz type estimates (on $\mathbb{R}$) and Bourgain type estimates (on $\mathbb{T}$)   low regularity estimates  could be recently established for resonance based filtered discretisations  for rough data \eqref{rough}; see \cite{ORSBourg,ORS19}. Compared to the filtered Lie splitting \eqref{scheme}  resonance based filtered schemes allow (thanks to their favorable  local error structure) improved convergence rates for solutions in $H^s$ with $s>1/4$. However, for very rough data, i.e.,  in regimes where $0 < s <1/4$, the convergence rate we show here for the filtered Lie splitting \eqref{scheme} coincides with the convergence rate established for the  resonance based  schemes. This is due to the fact that at this point of roughness the ``worst" regularity restrictions stem from the ``space discretization", which consists in projecting the equation
onto a finite number of Fourier modes,  and not from the local error of the time discretization.

Our error estimates are based on  discrete Bourgain type estimates developed in  \cite{ORSBourg}. It seems also possible to extend the analysis developed in this paper to higher dimensions and more general nonlinearities. The central task would lie in establishing the corresponding discrete counterpart of the continuous Bourgain estimates given in  \cite{Bour93a} in the various cases as done in  \cite{ORSBourg} for the periodic NLS.

Note that the filtered Lie splitting \eqref{scheme} can be seen as a classical Lie splitting discretisation of the projected equation
\begin{equation}
\begin{aligned}\label{nlsP}
i \partial_t  \uta & = - \partial_{x}^2 \uta - \mu  \Pi_{\taul}\left( \vert  \Pi_{\taul}   \uta \vert^2  \Pi_{\taul}  \uta\right), \qquad \uta(0) & =  \Pi_{\taul}  u(0).
\end{aligned}
\end{equation}
Hence, the idea is to first analyse the difference between the original NLS equation \eqref{nlsO} and its projected counterpart \eqref{nlsP} on the continuous level. This will then allow us to analyse  the time discretisation error introduced by the Lie splitting discretisation~\eqref{scheme} applied to the projected equation~\eqref{nlsP}.

\subsection*{Outline of the paper.} In a first step we study the difference between the original NLS equation~\eqref{nlsO} and its projected counterpart \eqref{nlsP}; see Section \ref{sec:errPro}. This will give a bound on
$$
\Vert u(t) -  \uta(t) \Vert_{L^2}.
$$
Then we will analyse the time discretisation error introduced by the Lie splitting discretisation~\eqref{scheme} applied to the projected equation~\eqref{nlsP}; see Section \ref{sec:errSplit}. This will yield a bound on
$$
\Vert \uta(t_n) - u^n\Vert_{L^2}.
$$
Combining the above two estimates eventually allows us to establish the desired global error estimate on $\Vert u(t_n) - u^n\Vert_{L^2}$
since
$$
\Vert u(t_n) - u^n\Vert_{L^2} \leq \Vert u(t_n) -  \uta(t_n) \Vert_{L^2} + \Vert \uta(t_n) - u^n\Vert_{L^2}.
$$

We carry out the error analysis on the torus $\mathbb{T}$; however, all estimates also hold for the full space case $\mathbb{R}$.
\subsection*{Notations}

We close this section with some notation that will be used throughout the paper.
For two expressions $a$ and $b$, we write $a \lesssim b$ whenever $a \leq C b$ holds with some constant $C>0$, uniformly in $\tau \in (0, 1]$.  We further write $a\sim b$ if $b\lesssim a \lesssim b$. When we want to emphasize that $C$ depends on an additional parameter $\gamma$, we write $a \lesssim_{\gamma} b$.

Further, we denote $\langle \,\cdot\, \rangle = ( 1 + | \cdot |^2)^{1 \over 2}$ and for sequences $(u_{n})_{n \in \mathbb{Z}} \in 
 X^\mathbb{Z}$ in a Banach space $X$ with norm $\|\cdot \|_{X}$, we use the notation
 $$ \|u_{n}\|_{l^p_{\tau}X}=\left( \tau \sum_{n} \|u_{n}\|_{X}^p \right)^{1 \over p}, \quad
  \|u_{n}\|_{l^\infty_{\tau}X} = \sup_{n \in \mathbb{Z}} \|u_{n}\|_{X}.$$

\section{Error between the  exact and the projected equation}\label{sec:errPro}
In this section we estimate  the difference between the solutions of the original NLS equation  \eqref{nlsO} and its projected counterpart \eqref{nlsP}.
Such an estimate was already established in \cite{ORSBourg}.

Let us recall the definition of Bourgain spaces. A tempered distribution $u(t,x)$ on $\mathbb{R} \times \mathbb{T}$ belongs to the Bourgain space $X^{s, b}$ if its following norm is finite
\begin{equation*}
\|u\|_{X^{s, b}}=\left(\int_{\mathbb{R}}\sum_{k \in \mathbb{Z}}\left(1+ |k|\right)^{2s}\left(1+| \sigma+ k^2|\right)^{2b}|\widetilde{u}\left(\sigma, k\right)|^2 \dd\sigma\right)^{\frac{1}{2}},
\end{equation*}
where $\widetilde{u}$ is the space-time Fourier transform of $u$:
$$
\widetilde{u}(\sigma, k)= \int_{\mathbb{R}\times \mathbb{T}} \e^{-i \sigma t - i k x} u(t,x) \, \dd t \dd x.
$$
We shall also use a localized version of this space.  For $I \subset \mathbb{R}$ being an open interval, we say that $u \in X^{s,b}(I)$ \black if $\|u\|_{X^{s,b}(I)}< \infty$, where
$$
\|u\|_{X^{s,b}(I)} = \inf\{\|\overline{u} \|_{X^{s,b}}, \, \overline{u}|_{I} = u  \}.
$$
When $I= (0, T)$ we will often simply use the notation $X^{s,b}(T)$. We refer for example to \cite{ORSBourg} Lem\-ma~2.1 for some useful properties of these spaces in this setting
(and to \cite{Bour93a} and \cite{Tao06} for more details).
A  particularly useful property we will exploit in the following is the embedding $X^{s, b} \subset \mathcal{C}(\mathbb{R}, H^s)$ for $b>1/2$.
Let us first recall the following classical well-posedness result for \eqref{nlsO}.
\begin{theorem}
\label{theoNLS}
For every $T>0$ and $u_{0}\in L^2$, there exists a unique solution $u$ of \eqref{nlsO} such that $u\in \mathcal{C}([0, T], L^2) \cap X^{0, {b}}(T)$ for  any  $b \in (1/2, 5/8)$. Moreover, if $u_{0} \in H^{s_{0}}$, $s_{0}> 0$, then $ u\in \mathcal{C}([0, T], H^{s_{0}}) \cap X^{s_{0}, b}(T)$.
\end{theorem}
We refer again to \cite{ORSBourg} for a sketch of the proof.

In a similar way, we get for the solution of \eqref{nlsP} the following estimate.
\begin{proposition}
\label{propNLSK}
For  $u_{0} \in H^{s_{0}}$,  $s_{0} \geq 0$, and $\tau \in (0, 1]$, there exists a unique solution $u_{\tau}$ of \eqref{nlsP} such that $u_{\tau} \in X^{s_{0}, b}(T)$ for $b\in(1/2,5/8)$ and every $T > 0$. Moreover, for every $T>0$, there exists $M_{T}$ such that for every $\tau \in (0, 1]$, we have the estimate
$$
\| u_{\tau} \|_{X^{s_{0}, b}(T)} \leq M_{T}.
$$
Furthermore, it holds that uniformly for $\tau \in (0, 1]$, we have for some $C_{T}>0$,
$$
\| u - u_{\tau} \|_{L^\infty((0,T); L^2)} \lesssim \| u - u_{\tau} \|_{X^{0, b}(T)} \leq  C_{T}  \tau^{s_{0} \over 2}.
$$
\end{proposition}
This statement is proven in \cite{ORSBourg}. It is a special case of  Proposition 2.4 and Corollary 2.6 in that article with $K= \tau^{- {1 \over 2}}$.

Moreover, if $s_{0}>0$, we have the following additional property:
\begin{cor}\label{corsb}
Let  $b \in (5/8, 1]$ and assume that  $s_{0}>0$, then for every    $0 \leq s_1 <s_{0},$ we also have that  for every $T > 0$ and uniformly in $\tau$,
$$
\|u_{\tau} \|_{X^{s_1,b}(T)} \leq M_{T}.
$$
\end{cor}
This is proven in \cite{ORSBourg} Corollary 2.8.

\section{Error representation of the time discretisation of the projected equation}\label{sec:errSplit}

In this section we derive an estimate on the time discretisation error introduced by the Lie splitting discretisation~\eqref{scheme} applied to the  projected equation~\eqref{nlsP}. This will give an estimate on
$$
\Vert \uta(t_n) - u^n\Vert_{L^2}.
$$
First we rewrite the filtered splitting \eqref{scheme} with the aid of the telescopic identity as follows
\begin{equation}\label{schemT}
\begin{aligned}
u^n =  \e^{i n  \tau \partial_x^2} \Pi_{\taul}   u(0)
+i\mu  \tau \sum_{k=0}^{n-1}   \e^{i (n-k)  \tau \partial_x^2} \Pi_{\taul}  \frac{\e^{i \mu \tau \vert \Pi_{\taul} u^k\vert^2}-1}{i \mu  \tau} \Pi_{\taul} u^k
\end{aligned}
\end{equation}
(see also \cite{Ignat11}). Next we compare the scheme \eqref{schemT} with the mild solution of \eqref{nlsP}.
Note that the mild solution of  \eqref{nlsP} reads with the notation $t_n = n \tau$
\begin{equation}
\begin{aligned}\label{mildiV}
 \vta(t_n) & = \e^{i n  \tau \partial_x^2} \Pi_{\taul}   u(0)
 +i \mu \int_0^{t_n} \e^{i   (t_n-s) \partial_x^2}  \Pi_{\taul}
\left(\vert \Pi_{\taul} \vta(s) \vert^2  \Pi_{\taul} \vta(s)\right) \dd s\\
&=  \e^{i n  \tau \partial_x^2} \Pi_{\taul}   u(0)
 +i \mu \sum_{k=0}^{n-1} \int_{t_k}^{t_{k+1}} \e^{i   (t_n -s) \partial_x^2}  \Pi_{\taul}
\left(\vert \Pi_{\taul} \vta(s) \vert^2  \Pi_{\taul} \vta(s)\right)   \dd s\\
 &=  \e^{i n  \tau \partial_x^2} \Pi_{\taul}   u(0)
 +i \mu \sum_{k=0}^{n-1} \int_{0}^{\tau } \e^{i   (t_n-t_k -s)  \partial_x^2}  \Pi_{\taul}
\left(\vert \Pi_{\taul} \vta(t_k+s) \vert^2  \Pi_{\taul} \vta(t_k+s)\right)  \dd s.
\end{aligned}
\end{equation}
Thus, by taking the difference of \eqref{mildiV} and \eqref{schemT} we obtain that
\begin{equation}\label{global}
\vta(t_n) - u^n  
= i \mu \tau \sum_{k=0}^{n-1}  \e^{i (n-k)  \tau \partial_x^2}
\left( \Phi_{\mathcal{N}}^\tau(\vta(t_k)) - \Phi_{\mathcal{N}}^\tau(u^k) \right) + i \mu  \sum_{k=0}^{n-1}  \e^{i (n-k)  \tau \partial_x^2} \mathcal{E}_{\text{loc}}(t_k,\tau,\vta)
\end{equation}
with the nonlinear flow
\begin{equation}\label{numFlow}
\Phi_{\mathcal{N}}^\tau(w) = \Pi_{\taul} \left( \frac{\e^{i\mu \tau \vert \Pi_{\taul} w \vert^2}-1}{i \mu \tau} \Pi_{\taul} w\right)
\end{equation}
and the local error
\begin{equation}\label{localErr}
\begin{aligned}
\mathcal{E}_{\text{loc}}(t_k,&\tau,\vta)
\\& =
\Pi_{\taul} \int_{0}^{\tau } \e^{- i  s   \partial_x^2}
\left(\vert \Pi_{\taul} \vta(t_k+s) \vert^2  \Pi_{\taul} \vta(t_k+s)\right)
 -   \frac{\e^{i \mu\tau \vert \Pi_{\taul} \vta(t_k) \vert^2}-1}{i \mu\tau} \Pi_{\taul} \vta(t_k)\, \dd s
 \\& =
\Pi_{\taul} \int_{0}^{\tau }\left( \e^{- i  s   \partial_x^2} -1\right)
 \left(\vert \Pi_{\taul} \vta(t_k+s) \vert^2  \Pi_{\taul} \vta(t_k+s)\right)   \dd s
 \\&\ \quad +
\Pi_{\taul} \int_{0}^{\tau }
 \vert \Pi_{\taul} \vta(t_k+s) \vert^2   \Pi_{\taul} \Big( \vta(t_k+s)
 -    \vta(t_k) \Big)\dd s
 \\&\ \quad +
\Pi_{\taul} \int_{0}^{\tau }
\left( \vert \Pi_{\taul} \vta(t_k+s) \vert^2  -\frac{\e^{i\mu \tau \vert \Pi_{\taul} \vta(t_k) \vert^2}-1}{i \mu\tau}   \right) \Pi_{\taul} \vta(t_k) \,\dd s\\
&
= \mathcal{E}_{1}(t_{k}) + \mathcal{E}_{2}(t_{k}) + \mathcal{E}_{3}(t_{k}),
\end{aligned}
\end{equation}
where by \eqref{mildiV}
\begin{equation}\label{vd}
\begin{aligned}
\vta(t_k+s) &-    \vta(t_k) \\
&=\left( \e^{i s \partial_x^2} -1\right) \vta(t_k) +i  \mu
\int_0^s  \e^{i  (s-\xi)  \partial_x^2}  \Pi_{\taul}\left( \vert \Pi_{\taul} \vta(t_k+\xi ) \vert^2  \Pi_{\taul} \vta(t_k+\xi ) \right) \dd\xi.
\end{aligned}
\end{equation}

\section{Discrete Bourgain spaces}

Before performing local error and stability estimates for rough data, we shall recall the main properties of the discrete
Bourgain spaces introduced in \cite{ORSBourg}.

For sequences of functions $(u_{n}(x))_{n \in \mathbb{Z}},$ we define the Fourier transform $\widetilde{u_{n}}(\sigma, k)$ by
$$
\mathcal F_{\tau,x}(u_n)(\sigma,k) =\widetilde{u_{n}} (\sigma, k)= \tau \sum_{m \in \mathbb{Z}} \widehat{u_{m}}(k) \,\e^{i m \tau \sigma}, \quad \widehat{u_{m}}(k)= {1 \over 2\pi} \int_{-\pi}^\pi u_{m}(x) \,\e^{-i k x}\dd x.
$$
Parseval's identity then reads
\begin{equation}\label{parseval}
\| \widetilde{u_{n}}\|_{L^2l^2}= \|u_{n}\|_{l^2_{\tau}L^2},
\end{equation}
where
$$
\| \widetilde{u_{n}}\|_{L^2l^2}^2 = \int_{-{\pi \over \tau}}^{\pi\over \tau} \sum_{k \in \mathbb{Z}}
|\widetilde{u_{n}}(\sigma, k)|^2 \dd \sigma, \quad
\|u_{n}\|_{l^2_{\tau}L^2}^2 = \tau \sum_{m \in \mathbb{Z}} \int_{-\pi}^\pi  |u_{m}(x)|^2 \dd x.
$$
We  define the discrete Bourgain spaces $X^{s,b}_\tau$ for $s\ge 0$, $b\in\mathbb R$, $\tau>0$ by
\begin{equation}\label{norm2}
\| u_n \|_{X^{s,b}_{\tau}} = \left\| \langle k \rangle^s \langle  d_{\tau}(\sigma -k^2)  \rangle^b \widetilde{u_n}(\sigma, k)  \right\|_{L^2l^2},
\end{equation}
where  $d_{\tau}(\sigma)=\frac{\e^{i \tau \sigma} - 1}\tau$.
Note that $d_{\tau}$ is $2\pi/\tau$ periodic and that uniformly in $\tau$, we have $|d_{\tau}(\sigma)| \sim | \sigma |$ for $|\tau \sigma | \leq \pi$.

For  $s \in \mathbb{R}$ and $b> 1/2$, we have that $X^{s,b}_{\tau} \subset l^\infty_{\tau}H^s$:
\begin{equation}\label{sobbourg}
\|u_{n}\|_{l^\infty_{\tau}H^s } \lesssim_{b} \| u_{n}\|_{X^{s,b}_{\tau}}.
\end{equation}
From the properties of the $d_{\tau}$ function, we also have that
\begin{equation}\label{bshift}
\sup_{\delta \in [-4, 4]} \| \e^{i n\tau \delta \partial_{x}^2} u_{n}\|_{X^{s, b}_{\tau}} \lesssim_{b} \|u_{n}\|_{X^{s,b}_{\tau}},
\end{equation}
\begin{equation}\label{shiftt}
\sup_{\delta \in [-4, 4]} \|\e^{i n \tau \delta } u_{n}\|_{X^{s, b}_{\tau}} \lesssim_{b} \|u_{n}\|_{X^{s, b}_{\tau}}
\end{equation}
 and that the discrete spaces satisfy the embeddings
\begin{equation}\label{embdisc1}
\|u_{n}\|_{X^{0, b}_{\tau}} \lesssim { 1 \over  \tau^{b-b'}} \|u_{n}\|_{X^{0, b'}_\tau}, \quad b \geq b'.
\end{equation}
Some useful more technical properties are gathered in the following lemma; see~\cite[Lemma 3.4]{ORS19}.
\begin{lemma}\label{bourgainfaciled}
For $\eta \in \mathcal{C}^\infty_{c}(\mathbb{R})$ and $\tau\in(0,1]$, we have that
\begin{align}
\label{bourg1} &\| \eta(n \tau)  \e^{in \tau \partial_{x}^2} f\|_{X^{s,b}_{\tau}} \lesssim_{\eta, b} \|f\|_{H^s}, \quad s \in \mathbb{R}, \, b \in \mathbb{R}, \, f \in H^s, \\
\label{bourg2} &\| \eta(n \tau)  u_{n}\|_{X^{s,b}_{\tau}} \lesssim_{\eta, b} \|u_{n}\|_{X^{s,b}_{\tau}}, \quad s \in \mathbb{R}, \, b \in \mathbb{R} , \, u_{n} \in X^{s,b}_{\tau},\\
\label{bourg3} &\left\| \eta\left(\frac{n\tau}T \right) u_{n} \right\|_{X^{s,b'}_{\tau}} \lesssim_{\eta, b, b'} T^{b-b'} \|u_{n}\|_{X^{s,b}_{\tau}}, \quad s \in \mathbb{R},  -{1 \over 2} <b' \leq b <{ 1 \over 2},\, 0< T = N \tau  \leq 1, \, N \geq 1.
\end{align}
In addition, for
$$
U_{n}(x)= \eta(n \tau) \tau \sum_{m=0}^n  \e^{i ( n-m ) \tau \partial_{x}^2}  u_{m}(x),
$$
we have
\begin{equation}
\label{bourg4}\|U_{n}\|_{X^{s,b}_{\tau}} \lesssim_{\eta, b} \|u_{n}\|_{X^{s, b-1}_{\tau}}, \quad s \in \mathbb{R}, \, b>1/2.
\end{equation}
We stress that all given estimates are uniform in $\tau$.
\end{lemma}

The crucial product property in the analysis of cubic NLS for these discrete spaces is given in the following lemma.
\begin{lemma}\label{prodd}
We have, uniformly for $\tau \in (0, 1]$,
\begin{equation}\label{prodd1}
\|\Pi_{\taul}u_{n} \|_{l^4_{\tau}L^4} \lesssim  \|u_{n}\|_{X^{0, {3\over 8}}_{\tau}}.
\end{equation}
This yields for any sequences $(u_{n})$, $(v_{n})$, $(w_{n})$ and for every $s_{1} \geq 0$,
\begin{equation}\label{prodd3}
\| \Pi_{\tau}\left( \Pi_{ \taul} u_{n} \Pi_{ \taul} v_{n} \Pi_{\taul} w_{n} \right) \|_{X^{s_{1}, - { 3 \over 8}}_{\tau}}
\lesssim    \|u_{n}\|_{X^{s_{1}, {3 \over 8}}_{\tau}}  \|v_{n}\|_{X^{s_{1}, {3 \over 8}}_{\tau}} \|w_{n}\|_{X^{s_{1}, {3 \over 8}}_{\tau}},
\end{equation}
again uniformly for $\tau \in (0, 1]$.
\end{lemma}

\begin{proof}
The estimate \eqref{prodd1} was proven in \cite{ORSBourg} (see Lemma 3.1), we refer to this paper for the proof.
Note that \eqref{prodd1} states that $\Pi_{\tau}$ is a continuous map from $X^{0, {3\over 8}}_{\tau}$
to $l^4_{\tau}L^4$.  We can then also deduce by duality  that the adjoint  $\Pi_{\tau}^*= \Pi_{\tau}$
is continuous from $l^{4\over 3}_{\tau}L^{4 \over 3}$ to  $X^{0, -{3\over 8}}_{\tau}$ with the same norm, which means that
\begin{equation}
   \label{dualbourg}
   \| \Pi_{\tau} u_{n}\|_{X^{0, -{3\over 8}}_{\tau}} \lesssim \|u_{n}\|_{l^{4\over 3}_{\tau}L^{4 \over 3}}.
\end{equation}
To deduce \eqref{prodd3}, we first use that thanks to \eqref{dualbourg} we have   for every $s_1 \geq 0$ \black
$$  
\| \Pi_{\tau}\left( \Pi_{ \taul} u_{n} \Pi_{ \taul} v_{n} \Pi_{\taul} w_{n} \right) \|_{X^{s_{1}, - { 3 \over 8}}_{\tau}}
\lesssim \left\| \langle \partial_{x} \rangle^{s_{1}} \left(\Pi_{ \taul} u_{n} \Pi_{ \taul} v_{n} \Pi_{\taul} w_{n} \right) \right\|_{l^{4 \over 3}_{\tau}L^{4 \over 3}}.
$$
Next, by using the generalized Leibniz rule and the H\"older inequality, we obtain that
$$  \| \Pi_{\tau}\left( \Pi_{ \taul} u_{n} \Pi_{ \taul} v_{n} \Pi_{\taul} w_{n} \right) \|_{X^{s_{1}, - { 3 \over 8}}_{\tau}}
     \lesssim \| \langle \partial_{x} \rangle^{s_{1}} \Pi_{\tau} u_{n} \|_{l^4_{\tau}L^4} \| \langle \partial_{x} \rangle^{s_{1}} \Pi_{\tau} v_{n} \|_{l^4_{\tau}L^4} \| \langle \partial_{x} \rangle^{s_{1}} \Pi_{\tau} w_{n} \|_{l^4_{\tau}L^4}
$$
and we conclude by using \eqref{prodd1} again.
\end{proof}

Finally, the estimates of Corollary \ref{corsb} for the continuous solution, gives the following property
for its restriction on the grid.
\begin{proposition}\label{propunifdisc}
Let $u_{\tau}$ be the solution of \eqref{nlsP} and define the sequence  $u_{\tau}^n= u_{\tau}(n\tau + t', x)$. Assume that $u_0 \in H^{s_{0}}$, $s_{0}> 0 $. Then, for every   $s_{1}$,  such that $0 \leq s_{1} <  s_{0}$,  we have that
$$
\sup_{t' \in [0, 4 \tau]} \|\eta(n\tau) u^n_{\tau}\|_{X^{s_{1}, {3 \over 8}}_{\tau}} \leq C_{T}.
$$
\end{proposition}

In the following we shall still denote by $u_{\tau}$ a function which belongs globally  to $X^{s_{0}, b}(\mathbb{R}\times \mathbb{T})$,
 coincides   on $[0, T]$ with the solution $u_{\tau}$ of \eqref{nlsP}  given by Proposition \ref{propNLSK}  and vanishes for $ t \geq 2T$. With this notation,
we get from the previous proposition that
$$
\sup_{t' \in [0, 4 \tau]} \| u_{\tau}(n\tau + t', x)  \|_{X^{s_{1}, {3 \over 8}}_{\tau}} \leq C_{T}.
$$

\section{Global error estimates}
Let $e^{k} = u_\tau(t_k)- u^k$ denote the global error for the modified equation and let $\eta$ be a smooth and compactly supported function, which is one on $[-1, 1]$ and supported in $[-2, 2]$.
Using the properties of $\eta$ allows us to rewrite \eqref{global} as
\begin{equation}
\label{eqen}
e^{n}= i \mu \tau \eta(t_{n})   \sum_{k=0}^{n-1}  \e^{i (n-k)  \tau \partial_x^2} \eta\left( {t_{k} \over T_{1}}\right)
\bigl( \Phi_{\mathcal{N}}^\tau(\vta(t_k)) - \Phi_{\mathcal{N}}^\tau(u_{\tau}(t_{k})-e^k)\bigr) + \mathcal{R}_{n}
\end{equation}
with
\begin{equation}
\label{Rn}
\mathcal{R}_{n}= i \eta(t_{n})\mu  \sum_{k=0}^{n-1}  \e^{i (n-k)  \tau \partial_x^2} \eta(t_{k}) \mathcal{E}_{\text{loc}}(t_k,\tau,\vta)
\end{equation}
for $0 \leq n \leq N_{1}$, where $N_{1}=\lfloor {T_{1}\over \tau}\rfloor$ with $T_{1} \leq \min(1, T)$.
The benefit if this modification is that the solution of \eqref{eqen} is globally defined, so that we can use global
Bourgain spaces to estimate it. \black

The aim of this section is the proof of the following estimate of the global error $ \mathcal{R}_{n}.$
\begin{proposition}
 \label{globalerror}
 For $b  \in (1/2, 5/8)$, $ s_{1} \in [0, 2]$, $0 \leq s_{1}< s_{0}$,  we have
 $$\|\mathcal{R}_{n}\|_{X^{0, b}_{\tau}}\leq C_{T}\tau^{s_{1} \over 2}.$$
\end{proposition}

 \begin{proof}
By using \eqref{bourg4}, we first obtain that
\begin{equation}
\label{estE0} \|\mathcal{R}_{n} \|_{X^{0, b}_{\tau}} \lesssim  \tau^{-1}
\| \mathcal{E}_{\text{loc}}(t_n,\tau,\vta)\|_{X^{0, b-1}_{\tau}} \lesssim
 \tau^{-1} \| \mathcal{E}_{\text{loc}}(t_n,\tau,\vta)\|_{X^{0, -{3 \over 8}}_{\tau}}.
 \end{equation}
 Thanks to \eqref{localErr}, we need to estimate $ \tau^{-1} \| \mathcal{E}_{i}\|_{X^{0, -{3 \over 8}}_{\tau}}$, $i=1, \, 2, \, 3$.

 To estimate $\mathcal{E}_{1}$, since $\e^{-is \partial_{x}^2}-1$ and $\Pi_{\taul}$ are Fourier multipliers
   in the space variable, and since  $\Pi_{\taul}$ projects on frequencies less than $\tau^{- {1 \over 2}},$ we observe that
  for any function $F(t_{n})$, we have
  $$ \sup_{s \in [- \tau, \tau]}\| (\e^{-is \partial_{x}^2}-1) \Pi_{\tau} F(t_{n}) \|_{X^{0, b}_{\tau}}
   \lesssim \tau^{s_{1}\over 2}  \| F(t_{n}) \|_{X^{s_{1}, b}_{\tau}}$$
   for $0 \leq s_{1} \leq 2$.
   Therefore, we get that
$$ \tau^{-1} \| \mathcal{E}_{1}(t_n)\|_{X^{0, -{3 \over 8}}_{\tau}}
 \lesssim \tau^{s_{1} \over 2} \sup_{s \in [0, \tau]} 
\left\| \Pi_{\tau}\left(\vert\Pi_{\taul} \vta(t_n+s) \vert^2  \Pi_{\taul} \vta(t_n+s)\right)\right\|_{X^{s_{1}, -{3 \over 8}}_{\tau}}.$$
This yields, thanks to Lemma \ref{prodd},
$$  \tau^{-1} \| \mathcal{E}_{1}(t_n)\|_{X^{0, -{3 \over 8}}_{\tau}}
 \lesssim \tau^{s_{1} \over 2} \sup_{s \in [0, \tau]} \|u_{\tau}(t_{n} + s)\|_{X^{s_{1}, {3 \over 8}}_{\tau}}^3.$$
Therefore, by using  Proposition \ref{propunifdisc}, we obtain that
\begin{equation}
\label{estE1} \tau^{-1} \| \mathcal{E}_{1}(t_n)\|_{X^{0, -{3 \over 8}}_{\tau}}
 \lesssim  C_{T}\tau^{s_{1} \over 2}.
 \end{equation}
 Next, by using \eqref{localErr} and \eqref{vd}, we get that
 $$  \tau^{-1} \| \mathcal{E}_{2}(t_n)\|_{X^{0, -{3 \over 8}}_{\tau}}
  \lesssim  \tau^{-1} \| \mathcal{E}_{2, 1}(t_n)\|_{X^{0, -{3 \over 8}}_{\tau}}
  + \tau^{-1} \| \mathcal{E}_{2, 2}(t_n)\|_{X^{0, -{3 \over 8}}_{\tau}}, $$
  where
  \begin{align*}
  \mathcal{E}_{2, 1}(t_n)
  & = \Pi_{\taul} \int_{0}^{\tau }
 \vert \Pi_{\taul} \vta(t_n+s) \vert^2   \Pi_{\taul} \Big(( \e^{is \partial_{x}^2}- 1) u_{\tau}(t_{n})\Big)\dd s
 , \\
   \mathcal{E}_{2, 2}(t_n)
  & = i \mu \Pi_{\taul} \int_{0}^{\tau }
 \vert \Pi_{\taul} \vta(t_n+s) \vert^2 \,  \Pi_{\taul}\!
\int_0^s  \e^{i  (s-\xi)  \partial_x^2}  \Pi_{\taul}\left( \vert \Pi_{\taul} \vta(t_n+\xi ) \vert^2  \Pi_{\taul} \vta(t_n+\xi ) \right) \dd\xi\dd s.
\end{align*}
By using again \eqref{prodd3}, we obtain that
$$  \tau^{-1} \| \mathcal{E}_{2, 1}(t_n)\|_{X^{0, -{3 \over 8}}_{\tau}}
 \lesssim   \sup_{s \in [0, \tau]}\left( \|u_{\tau}(t_{n} + s)\|_{X^{0, {3 \over 8}}_{\tau}}^2
 \| ( \e^{is \partial_{x}^2}- 1) \Pi_{\taul}u_{\tau}(t_{n})\|_{X^{0, {3 \over 8}}_{\tau}}\right).$$
  Since, for $s\in [0, \tau]$, we have
  $$  \| ( \e^{is \partial_{x}^2}- 1) \Pi_{\taul}u_{\tau}(t_{n})\|_{X^{0, {3 \over 8}}_{\tau}}
   \lesssim \tau^{s_{1} \over 2} \|u_{\tau}(t_{n})\|_{X^{s_{1}, {3 \over 8}}_{\tau}}, $$
   we get by using again Proposition \ref{propunifdisc} that
   $$  \tau^{-1} \| \mathcal{E}_{2, 1}(t_n)\|_{X^{0, -{3 \over 8}}_{\tau}}
    \leq C_{T} \tau^{s_{1} \over 2}.$$
In a similar way, we can first obtain thanks to \eqref{prodd3} that
\begin{align*}
\tau^{-1} \| \mathcal{E}_{2, 1}&(t_n)\|_{X^{0, -{3 \over 8}}_{\tau}}\\
&\lesssim   \sup_{s \in [0, \tau]}\left( \|u_{\tau}(t_{n} + s)\|_{X^{0, {3 \over 8}}_{\tau}}^2
  \left\|\int_0^s  \e^{i  (s-\xi)  \partial_x^2}  \Pi_{\taul}\left( \vert \Pi_{\taul} \vta(t_n+\xi ) \vert^2  \Pi_{\taul} \vta(t_n+\xi ) \right) \dd\xi \right\|_{X^{0, {3 \over 8}}_{\tau}}\right).
\end{align*}
  Next, by using successively \eqref{bshift},  \eqref{embdisc1} and \eqref{prodd3}, we write
  \begin{align*}
  \Bigl\|\int_0^s  \e^{i  (s-\xi)  \partial_x^2}  \Pi_{\taul}&\left( \vert \Pi_{\taul} \vta(t_n+\xi ) \vert^2  \Pi_{\taul} \vta(t_n+\xi ) \right) \dd\xi \Bigr\|_{X^{0, {3 \over 8}}_{\tau}}\\
&\lesssim  \tau \sup_{\xi \in [0, \tau]}  \left\| \vert \Pi_{\taul} \vta(t_n+\xi ) \vert^2  \Pi_{\taul} \vta(t_n+\xi ) \right\|_{X^{0, {3 \over 8}}_{\tau}}\\ 
&\lesssim \tau^{1 \over 4}
 \sup_{\xi \in [0, \tau]}  \left\| \vert \Pi_{\taul} \vta(t_n+\xi ) \vert^2  \Pi_{\taul} \vta(t_n+\xi ) \right\|_{X^{0, -{3 \over 8}}_{\tau}}
 \lesssim \tau^{1 \over 4}
 \sup_{\xi \in [0, \tau]}  \| \vta(t_n+\xi ) \|_{X^{0, {3 \over 8}}_{\tau}}^3.
  \end{align*}
Consequently, by using again Proposition \ref{propunifdisc}, we find that
$$ \tau^{-1} \| \mathcal{E}_{2, 2}(t_n)\|_{X^{0, -{3 \over 8}}_{\tau}}
    \leq C_{T} \tau^{1 \over 4}$$
    and hence that
    $$  \tau^{-1} \| \mathcal{E}_{2}(t_n)\|_{X^{0, -{3 \over 8}}_{\tau}}
    \leq C_{T} (\tau^{s_{1} \over 2} + \tau^{1 \over 4} ) \leq C_{T} \tau^{s_{1} \over 2}$$
    if $s_{1} \leq 1/2$.

For $s_{1} >1/2$, we can estimate $\mathcal{E}_{2, 2}$ in a different way. We just use that
\begin{multline*}
\tau^{-1}  \| \mathcal{E}_{2, 2}(t_n)\|_{X^{0, -{3 \over 8}}_{\tau}}
 \lesssim  \tau^{-1}  \| \mathcal{E}_{2, 2}(t_n)\|_{X^{0, 0}_{\tau}}
 \\  \lesssim    \sup_{s \in [0, \tau]}  \|u_{\tau}(t_{n} + s)\|_{l^\infty_{\tau}L^8}^2
  \left\|\int_0^s  \e^{i  (s-\xi)  \partial_x^2}  \Pi_{\taul}\left( \vert \Pi_{\taul} \vta(t_n+\xi ) \vert^2  \Pi_{\taul} \vta(t_n+\xi ) \right) \dd\xi \right\|_{l^\infty_{\tau}L^4}
\end{multline*}
and we employ the Sobolev embedding $H^{s_{1}}\subset L^p(\mathbb{T} )$ for every $p$ together with the fact that
 $  H^{s_{1}}$ is an  algebra to obtain that
 $$ \tau^{-1}  \| \mathcal{E}_{2, 2}(t_n)\|_{X^{0, -{3 \over 8}}_{\tau}}  \leq C_{T} \tau.$$
Combining these estimates, we finally get
 \begin{equation}
 \label{estE2}  \tau^{-1} \| \mathcal{E}_{2}(t_n)\|_{X^{0, -{3 \over 8}}_{\tau}}
    \leq C_{T} \tau^{s_{1} \over 2}
    \end{equation}
 for every $0 \leq s_{1} \leq 2.$

 It remains to estimate $\mathcal{E}_{3}.$ We can rewrite it as
 $$ \mathcal{E}_{3}(t_{n})=   \mathcal{E}_{3, 1}(t_{n}) +  \mathcal{E}_{3, 2}(t_{n})$$
  with
  \begin{align*}
  \mathcal{E}_{3, 1}(t_{n}) &= \Pi_{\taul} \int_{0}^{\tau }
\left( \vert \Pi_{\taul} \vta(t_n+s) \vert^2 -   \vert \Pi_{\taul} \vta(t_n) \vert^2 \right) \, \Pi_{\taul} \vta(t_n) \,\dd s \\
 & =   \Pi_{\taul} \int_{0}^{\tau }
 \mbox{Re} \left(  (\Pi_{\taul} \vta(t_n+s) - \Pi_{\tau}\vta(t_{n} ) )(\overline{ \Pi_{\taul} \vta(t_n+s)} +\overline{\Pi_{\tau}\vta(t_{n} )})\right) \,    \Pi_{\taul} \vta(t_n) \,\dd s  \\
 \mathcal{E}_{3, 2}(t_{n})  &= - \tau  \Pi_{\taul} \left(
\frac{\e^{i\mu \tau \vert \Pi_{\taul} \vta(t_n) \vert^2}-1 -i \mu \tau  \vert \Pi_{\taul} \vta(t_n) \vert^2}{i \mu\tau}  \Pi_{\taul} \vta(t_n)
\right).
   \end{align*}
 By using again \eqref{vd}, the estimate of  $ \mathcal{E}_{3, 1}(t_{n})$ is similar to the estimate of $ \mathcal{E}_{2}(t_{n})$. We find again that
  $$ \tau^{-1} \| \mathcal{E}_{3, 1}(t_n)\|_{X^{0, -{3 \over 8}}_{\tau}}
    \leq C_{T} \tau^{s_{1} \over 2}.$$
    For $ \mathcal{E}_{3, 2}$, by using successively \eqref{dualbourg} and the estimate
    $$ \left| {\e^{i \mu \tau x}- 1 - i\mu  \tau x \over i \mu  \tau}  \right| \lesssim \tau |x|^2, \quad \forall x \in \mathbb{R},$$
     we get that
  $$ \tau^{-1} \| \mathcal{E}_{3, 2}(t_n)\|_{X^{0, -{3 \over 8}}_{\tau}}
   \lesssim \left\| \frac{\e^{i\mu \tau \vert \Pi_{\taul} \vta(t_n) \vert^2}-1 -i \mu \tau  \vert \Pi_{\taul} \vta(t_n) \vert^2}{i \mu\tau}  \Pi_{\taul} \vta(t_k)  \right\|_{l^{4\over 3}_{\tau}L^{4\over 3}}
    \lesssim  \tau  \left\|  \Pi_{\taul} \vta(t_n) \right\|_{l^{{20 \over 3}}_{\tau}L^{20 \over 3}}^5.$$
    Next, by using successively  that $W^{ {1 \over 10}, 4}\subset L^{20 \over 3}(\mathbb{T})$ and \eqref{prodd1}, we get that
  $$ \tau^{-1} \| \mathcal{E}_{3, 2}(t_n)\|_{X^{0, -{3 \over 8}}_{\tau}}
   \lesssim_{T} \tau \left\| \Pi_{\tau}  \vta(t_n) \right\|_{l^4_{\tau}W^{{1 \over 10}, 4} }
    \lesssim \tau \left\| \Pi_{\tau}  \vta(t_n) \right\|_{X^{ {1 \over 10}, {3 \over 8}}_{\tau}}.$$
  Consequently, if $2 \geq s_{1} >1/10$, we get from Proposition \ref{propunifdisc} that
$$  \tau^{-1} \| \mathcal{E}_{3, 2}(t_n)\|_{X^{0, -{3 \over 8}}_{\tau}}  \lesssim \tau C_{T} \lesssim \tau^{s_{1} \over 2} C_{T}.$$
 If $ s_{1} \leq 1/10$, we can write that
 $$ \tau^{-1} \| \mathcal{E}_{3, 2}(t_n)\|_{X^{0, -{3 \over 8}}_{\tau}}  \lesssim
    \tau \left\| \Pi_{\tau}  \vta(t_n) \right\|_{X^{ {1 \over 10}, {3 \over 8}}_{\tau}}
     \lesssim \tau^{ {19 \over 20} + {s_{1} \over 2}} C_{T} \lesssim \tau^{{s_{1}\over 2}}C_{T}.$$
     Consequently, we have also obtained that
  \begin{equation}
  \label{estE3}\tau^{-1} \| \mathcal{E}_{3}(t_n)\|_{X^{0, -{3 \over 8}}_{\tau}}
    \leq C_{T} \tau^{s_{1} \over 2}.
    \end{equation}
We end the proof gathering \eqref{estE0}, \eqref{estE1}, \eqref{estE2} and \eqref{estE3}.\end{proof}

\section{Proof of Theorem \ref{maintheo}}
We are now in a position to give the proof of Theorem \ref{maintheo}.
We first observe that thanks to Proposition \ref{propNLSK}, we have from the triangle inequality that
\begin{equation}
\label{erreurtotale}
\|u(t_{n})- u^n\|_{L^2} \leq  \|u(t_{n}) - u_{\tau}(t_n)\|_{L^2} + \|u_{\tau}(t_{n}) - u^n\|_{L^2} \leq C_{T}\tau^{s_{1}   \over 2} + \| e^n\|_{l^\infty_{\tau}L^2},
\end{equation}
where $e^{n}$ solves \eqref{eqen}. To get the error estimates of Theorem \ref{maintheo}, it thus suffices to estimate $ \|e^n\|_{X_{\tau}^{0,b}}$ for some $b\in (1/2, 5/8)$  thanks to \eqref{sobbourg}.

By using equation \eqref{eqen},  Lemma \ref{bourgainfaciled} and the error estimate of Proposition \ref{globalerror},  we get  that
\begin{equation}
\label{preuveerreur1}
\|e^n \|_{X^{0, b}_{\tau}} \leq   C_T T_{1}^{\varepsilon_{0}} \|\Phi_{\mathcal{N}}^\tau (u_{\tau}(t_{n}))-
 \Phi_{\mathcal{N}}^\tau (u_{\tau}(t_{n})-e^n)\|_{X^{0, - { 3 \over 8}}_{\tau}} +  C_{T}\tau^{ s_{1}\over 2}.
\end{equation}
Here, we further employed~\eqref{bourg3} with $T_1$ still to be determined and $\varepsilon_0=1-b-\frac38>0$. \black
Note that, by using \eqref{numFlow}, we can write
\begin{align*}
&\Phi_{\mathcal{N}}^\tau (u_{\tau}(t_{n}))-
 \Phi_{\mathcal{N}}^\tau (u_{\tau}(t_{n})-e^n)=
 \Pi_{\tau}\left( \Psi_{\mathcal{N}}^\tau (u_{\tau}(t_{n}))-
 \Psi_{\mathcal{N}}^\tau (u_{\tau}(t_{n})-e^n)\right), \\
&\Psi_\mathcal{N}^\tau(w)=
  \frac{\e^{i\mu \tau \vert \Pi_{\taul} w \vert^2}-1}{i \mu \tau} \Pi_{\taul} w,
\end{align*}
where we have, uniformly in $n$ and $\tau$, the pointwise estimate
\begin{equation}
\label{phipointwise}
 \left| { \Psi}_{\mathcal{N}}^\tau (u_{\tau}(t_{n}))-
{ \Psi}_{\mathcal{N}}^\tau (u_{\tau}(t_{n})-e^n)\right|
  \lesssim |\Pi_{\tau}u_{\tau}(t_{n})|^2 | \Pi_{\tau}e^{n}| + |\Pi_{\tau}u_{\tau}(t_{n})| |\Pi_{\tau}e^n|^2  + |\Pi_{\tau}e^n|^3.
  \end{equation}
  We can then  get by using \eqref{dualbourg} that
$$\|\Phi_{\mathcal{N}}^\tau (u_{\tau}(t_{n}))-
 \Phi_{\mathcal{N}}^\tau (u_{\tau}(t_{n})-e^n)\|_{X^{0, - { 3 \over 8}}_{\tau}}
  \lesssim \|{ \Psi}_{\mathcal{N}}^\tau (u_{\tau}(t_{n}))-
{  \Psi}_{\mathcal{N}}^\tau (u_{\tau}(t_{n})-e^n)\|_{l^{4 \over 3}_{\tau}L^{4\over 3}}
$$
and hence, thanks to \eqref{phipointwise}
and H\"older's inequality, that
\begin{multline*}\|\Phi_{\mathcal{N}}^\tau (u_{\tau}(t_{n}))-
 \Phi_{\mathcal{N}}^\tau (u_{\tau}(t_{n})-e^n)\|_{X^{0, - { 3 \over 8}}_{\tau}}
  \\\lesssim \|\Pi_{\tau}u_{\tau}(t_{n})\|_{l^4_{\tau}L^4}^2 \|\Pi_{\tau} e^n\|_{l^4_{\tau}L^4} +    \|\Pi_{\tau}u_{\tau}(t_{n})\|_{l^4_{\tau}L^4}
  \|\Pi_{\tau} e^n\|_{l^4_{\tau}L^4}^2 + \|\Pi_{\tau} e^n\|_{l^4_{\tau}L^4}^3.
  \end{multline*}
  Next, by using again \eqref{prodd1} and Proposition \ref{propunifdisc}, we finally get that
  $$  \|\Phi_{\mathcal{N}}^\tau (u_{\tau}(t_{n}))-
 \Phi_{\mathcal{N}}^\tau (u_{\tau}(t_{n})-e^n)\|_{X^{0, - { 3 \over 8}}_{\tau}}
  \leq C_{T}( \|e^n \|_{X^{0, {3 \over 8}}_{\tau}} + \|e^n \|_{X^{0, {3 \over 8}}_{\tau}}^2 + \|e^n \|_{X^{0, {3 \over 8}}_{\tau}}^3).$$
  This yields by using \eqref{preuveerreur1}
 $$ \|e^n \|_{X^{0, b}_{\tau}} \leq   C_T T_{1}^{\varepsilon_{0}}
 \left(  \|e^n \|_{X^{0, {3 \over 8}}_{\tau}} + \|e^n \|_{X^{0, {3 \over 8}}_{\tau}}^2 + \|e^n \|_{X^{0, {3 \over 8}}_{\tau}}^3\right) +
   C_{T}\tau^{ s_{1}\over 2}.$$
   By choosing $T_{1}$ sufficiently small we thus get that
$$
\|e^n \|_{X^{0, b}_{\tau}} \leq  C_{T}\tau^{s_{1} \over 2}.
$$
This proves the desired estimate \eqref{final1} for $ 0 \leq n \leq N_{1}= T_{1}/ \tau$. We can then iterate in a classical way the argument on  $ T_{1}/ \tau \leq  n \leq 2 T_{1}/\tau$ and so on to get the final estimate for $ 0 \leq n \leq T/\tau$.

\section{Numerical experiments}

In this section we numerically underline our convergence Theorem  \ref{maintheo}. We illustrate the convergence order of the filtered splitting method in the case of rough and smooth initial data. In Figure~\ref{fig} we solve the periodic Schr\"{o}dinger equation \eqref{nlsO} with the filtered Lie splitting \eqref{scheme} and, respectively, the  filtered Strang splitting with initial values
\begin{align*}
u(0) \in H^s\quad \text{with } s = \frac{1}{2}\text{ and }  s = 4;
\end{align*}
see, e.g., \cite{KOS19} for details on the construction of  rough initial values.   We employ a standard Fourier pseudospectral method for the discretization in space and we choose as largest Fourier mode $K = 2^{14}$, i.e., the spatial mesh size $\Delta x =  0.00038$.  As a reference solution we use the Strang splitting method with $K$ spatial points and a very small time step size $\tau$.   Then, for each time step $\tau$ we measure the error between the filtered Lie and filtered Strang splitting method and the projected reference solution by employing $\Pi_\tau$   in the corresponding discrete $L^2$ norm. Our numerical findings confirm their convergence rate of order $\mathcal{O}(\tau^{1/4})$ for solutions in $H^{1/2}$ (see Theorem  \ref{maintheo}) as well as their expected full order of convergence (order one for Lie and order two for Strang) in the case of smooth solutions.
\begin{figure}[h!]
\centering
\includegraphics[width=0.471\linewidth]{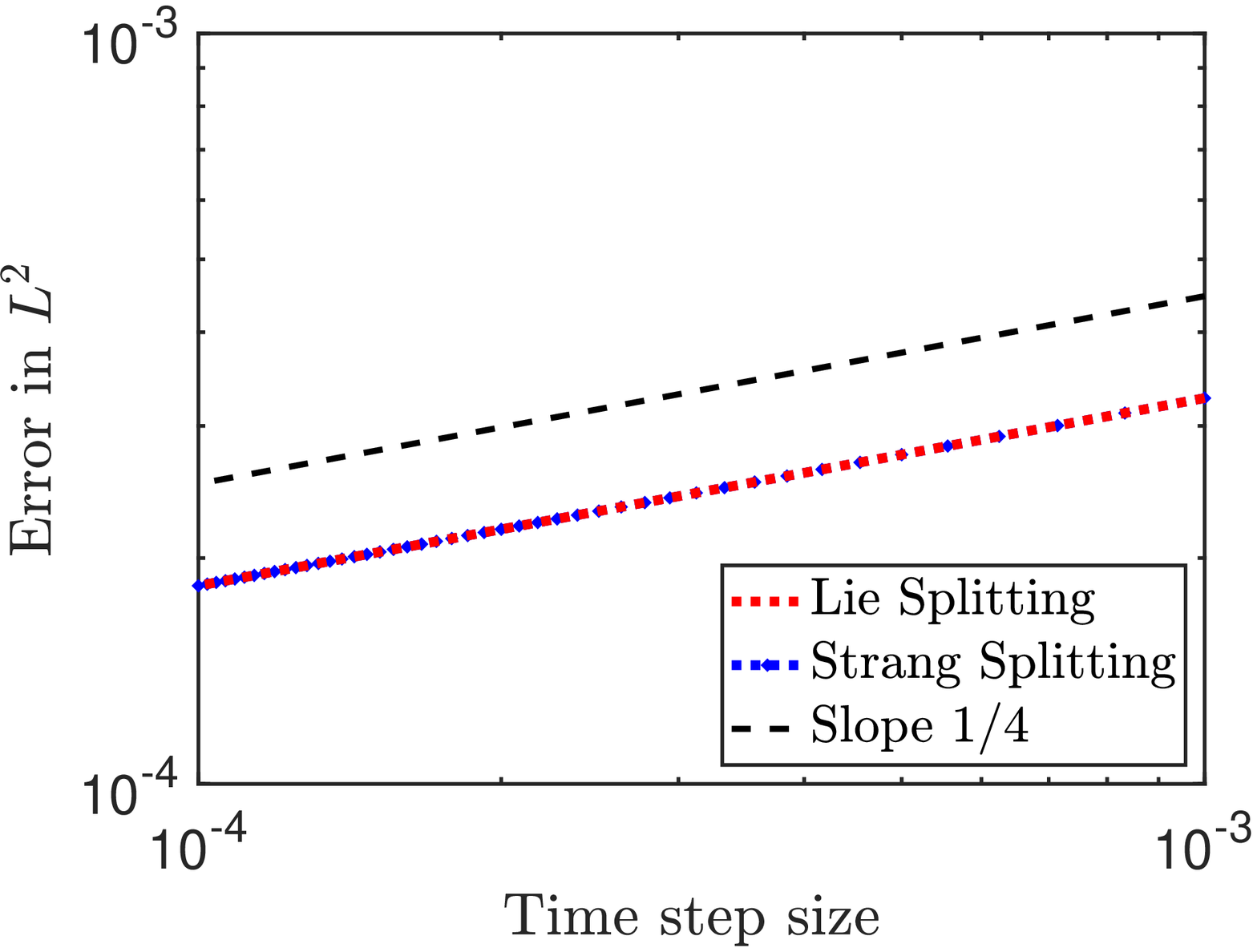}
\hfill
\includegraphics[width=0.48\linewidth]{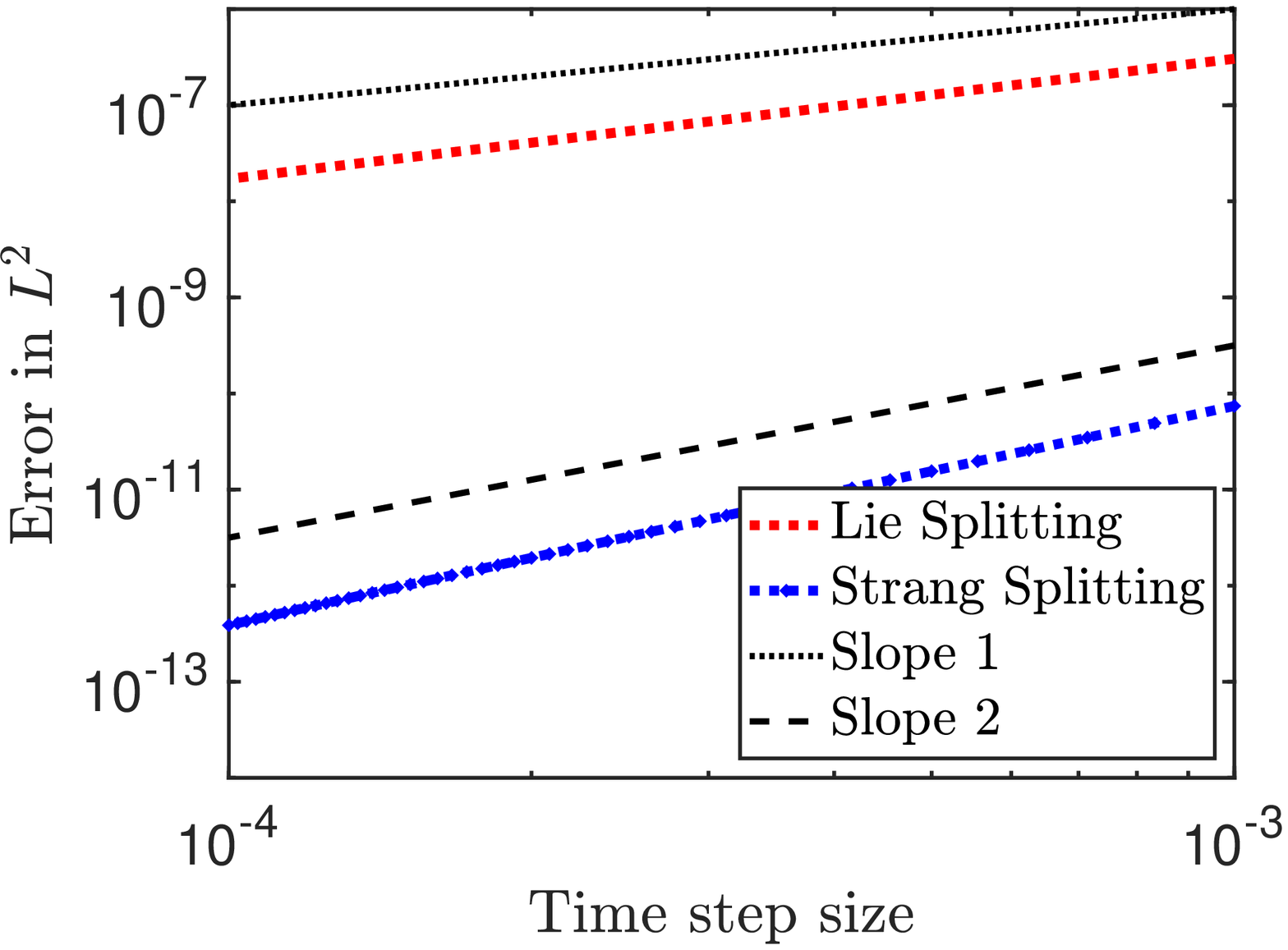}
\caption{$L^2$ error of the   Lie and Strang splitting schemes for rough and smooth initial data.  Left:  rough initial value $u(0) \in H^{1/2}$; right:  smooth initial value  $u(0) \in H^4$.}\label{fig}
\end{figure}

%
%
%

\subsection*{Acknowledgements}

{\small
KS has received funding from the European Research Council (ERC) under the European Union's Horizon 2020 research and innovation programme (grant agreement No. 850941).
}

\end{document}